\documentclass[12pt]{article}

\usepackage{latexsym,amstext,amsmath,amssymb,amsthm, amsopn}

\usepackage[english]{babel}
\usepackage[latin1]{inputenc}
\usepackage[OT1]{fontenc}

\newtheorem{thm}{Theorem}

\newtheorem{cor}[thm]{Corollary}
\newtheorem{prop}[thm]{Proposition}

\DeclareMathOperator{\sgn}{sgn}
\DeclareMathOperator{\RE}{Re}
\DeclareMathOperator{\IM}{Im}

\def \EE {\mathbb{E}}

\def \RR {\mathbb{R}}
\def \QQ {\mathbb{Q}}
\def \CC {\mathbb{C}}
\def \PP {\mathbb{P}}
\def \pS {\mathcal{S}}
\def \pL {\mathcal{L}}
\renewcommand{\Re}{{\rm Re}}
\renewcommand{\Im}{{\rm Im}}
\def \Log {{\rm Log}}
\def \Arg {{\rm Arg}}
\def \hkappa {\widehat{\kappa}}

\title{On exit time of stable processes}
\author{Piotr Graczyk, Tomasz Jakubowski}
\date{\today}
\begin{document}
\maketitle

\footnotetext{\noindent
\emph{ 2000 Mathematics Subject Classification:} 60G52   \\
\emph{Key words and phrases:} Stable process, exit time \\
This research was partially supported by grants MNiSW and ANR-09-BLAN-0084-01.}

\begin{abstract}
 We study the exit time $\tau=\tau_{(0,\infty)}$ for 1-dimensional  strictly stable processes
and  express its Laplace transform at $t^\alpha$ as the Laplace transform  of a positive
random variable with explicit density.  Consequently,  $\tau$ satisfies
some  multiplicative convolution relations. For some stable processes, e.g. for the symmetric
$\frac23$-stable process, explicit formulas for the Laplace transform and the density
of $\tau$ are obtained as an application.
\end{abstract}
\section{Introduction}
Let $\alpha \in (0,2)$ and $(X_t, \PP^x)$ be a strictly $\alpha$-stable process in $\RR$ with characteristic function
$$
\EE^0 e^{i X_t z} = \exp\left[-t |z|^\alpha\left(1-i \beta \arctan \frac{\pi \alpha}{2}\sgn z\right)\right],
$$
where $\beta \in [-1,1]$ and $\beta =0$ for $\alpha=1$. For any $D \subset \RR$ let 
$$\tau_D = \inf\{t \ge 0 \colon X_t \not\in D\}$$
be the first exit time from $D$ of the process $X_t$.  Throughout this article we shall consider the starting point $x>0$ and
$$
\tau=\tau_{(0,\infty)},
$$
the exit time of $X_t$ from the positive half-line. 

The question of the first exit time from domains are basic for all stochastic processes. Surprisingly  few exact formulas are known for stable processes. The only exceptions are Brownian motion, completely asymmetric stable processes with $\alpha >1$ (see \cite{Bing}, \cite{2006-EAK-sv}, \cite{TS-em})
and symmetric Cauchy process (\cite{1956-DAD-tams}, see also \cite{KKMS}). The quotient $\hat \tau/\tau$
was studied for independent $\tau$ and dual $\hat \tau$ in \cite{Doney2}. 

Some recent results on this problem in the completely asymmetric case  were obtained by T. Simon in \cite{TS-em} 
  and next were applied in
 \cite{TS-prep1} and   \cite{TS-prep2}.
On the other hand, some formulas were found by A. Kuznetsov in \cite{AK-ap}, however the final expressions are complicated. M. Kwa\'snicki in \cite{MK} gives an integral representation of the density of $\tau$ in the case of symmetric stable processes ($\beta =0$).

In this article we study the exit time $\tau=\tau_{(0,\infty)}$ for 1-dimensional  stable processes
and give in Theorem \ref{thm:main} a new formula for its Laplace transform. It follows that $\tau$ satisfies
some multiplicative convolution relations(Corollary \ref{cor:sploty}); in particular for $\alpha>1$
the exit time $\tau$ is the multiplicative convolution of a $1/\alpha$-stable subordinator
with an explicitly given random variable $M_{\alpha,\rho}$.  We generalize in this way
the result of  \cite{TS-em}   for all stable processes. Applications of Theorem \ref{thm:main} 
are next given  in the final part of the article. New explicit formulas for the Laplace transform and the density
of $\tau$ are proven for the processes dual to those of the Doney's class ${C}_{1,1}$, in particular 
for the symmetric  $\frac23$-stable process (Proposition \ref{cor:Lap} and Corollary \ref{cor:density}).
Further applications of Theorem \ref{thm:main} will be  presented in a forthcoming paper. 

The main tool  to prove the results of this  article  is a series representation
that we obtained in \cite{PG-TJ-aihp} for the logarithm of the bivariate Laplace exponent $\kappa(\eta,\theta)$ of the
ascending ladder process built from the process $X_t$. This
application of the series representation of $\ln \kappa$ was announced in \cite{PG-TJ-aihp}. 
It allows to  determine explicitely
in Proposition \ref{prop:St} the inverse Stieltjes transform of the function $1/\kappa(1,\theta)$. 
 

 \section{ Stieltjes transform and Wiener-Hopf factors}
 In this part of the article we will exploit our series representation
of
$\kappa(1,\theta)$  from \cite{PG-TJ-aihp} in inverting    a Stieltjes transform.

 Recall that if  $ \mu$ is a positive Borel measure on $[0,\infty)$
 then for  any $x \in (0,\infty)$  the Stieltjes transform of $\mu$ is defined by
\begin{equation}\label{eq:ST}
\pS \mu(\theta) = \int_0^\infty \frac{1}{\theta+x}   d\mu(x)
\end{equation}
 whenever the integral converges. 
 According to \cite{CB-LNM}, a function $G$ on $(0,\infty)$ is of the form 
$ G(\theta)=a+\pS \mu(\theta)$ for a positive measure $\mu$ and $ a\ge 0$
if and only if

(S1) $G$  extends to a  holomorphic function in the cut plane $\CC\setminus \RR_-$

(S2) $G(\theta)\ge 0$ for $\theta>0$

(S3) $\Im G(z)\le 0$ for $\Im z>0$.\\
Then the inverse Stieltjes transform is 
$$
\pS^{-1}(G)(x)= -(1/\pi)\lim_{y\to 0^+} \Im G(-x+iy)\,, \qquad x>0\,, 
$$
where the limit, in general, is in  vague sense  and equals   $\mu$. If $\mu$
is absolutely continuous with a continuous density, the limit is  equal to the density
of $\mu$ for all $x>0$ (\cite{Widder}).\\

Let $\alpha \in (0,2)$ and $(X_t, \PP^x)$ be a strictly $\alpha$-stable process in $\RR$. By $\kappa_{\alpha,\rho}(\eta,\theta)$ we denote the bivariate Laplace exponent of the
ascending ladder process built from $X_t$. We normalize it requiring that $\kappa_{\alpha,\rho}(1,0)=1$. To simplify the notation we will write $\kappa(\eta,\theta)$ for a fixed pair $\alpha,\rho$ (or equivalently a fixed process $X_t$). By $\hat\kappa$ we denote the Laplace exponent for the dual  process $\hat X_t=-X_t$. As usually we write the positivity coefficient 
$$
\rho=\PP^0(X_t\ge 0)= \frac{1}{2}+ \frac{1}{\pi \alpha}\arctan\left(\beta \tan\frac{\pi \alpha}{2}\right).
$$
\begin{prop}\label{prop:St}
For $\rho \in (0, 1] \setminus \{ 1/\alpha \}$ we have
\begin{equation}\label{eq:inv_kappa}
\frac{\sin(\rho \alpha  \pi)}{\pi} \int_0^\infty \frac{1}{x+\theta} \frac{x^{\alpha} \hkappa(1,x)}{x^{2\alpha} +2x^\alpha\cos(\rho\alpha\pi) +1}\, dx = \frac{1}{\kappa(1,\theta)}.
\end{equation}
\end{prop}
\begin{proof}

Denote $G(\theta)=1/\kappa(1,\theta)$. The function $G(\theta) $  extends to 
a  holomorphic function $h_1(z)$ on $\CC \setminus \RR_-$ (see \cite{2006-SF-ptrf},(i) p.205).
Let $\pL$ be the set of Liouville numbers. For $\theta \in (0,1)$ and $\alpha\not\in\pL \cup \QQ$ we have  by \cite{PG-TJ-aihp}
\begin{align*} 
 G(\theta) = \exp\left( - \sum_{m=1}^\infty \frac{(-1)^{m+1} \theta^{m}\sin(\rho m \pi)}{m\sin(\frac{m\pi}{\alpha})} - \sum_{k=1}^\infty \frac{(-1)^{k+1} \theta^{\alpha k}\sin(\rho\alpha k\pi)}{k\sin(\alpha k\pi)}\right).
\end{align*}
The right hand side of the last formula may be extended to a holomorphic  function $h_2(z)$ on $\{z \in \CC \colon |z| < 1\}\setminus \RR_-$ defining
$
w^{\alpha} = \exp(\alpha \Log w)
$
where 
$ \Log w=\ln|w|+i\Arg w$, $\Arg w\in(-\pi,\pi]$, is the principal value of the complex
logarithm.
We note that $h_1 = h_2$ on $(0,1)$, hence $h_1 = h_2$ on $\{z \in \CC \colon |z| < 1\}\setminus \RR_-$ and $h_2$ extends to a holomorphic function on $\CC \setminus \RR_-$, equal for $|z|>1$ to the holomorphic extension
of $ \frac{1}{\kappa(1,\theta)}$  for $\theta>1$.\\

In the first part of the proof we will compute
$$l(x)= -\frac1\pi \lim_{y\to 0^+ } \Im G(-x+iy)$$
for positive $x$.
Denote by $h(z)$ the expression under exponential of $h_2$.  Let us compute  for $0<x<1$
 $$
 l(x)= -\frac1\pi \lim_{y\to 0^+ }\Im  \exp(h(-x+iy))
= - \frac1\pi e^{\Re(w)} \sin({\Im(w)}),
$$
where 
$$
w=- \sum_{m=1}^\infty \frac{(-1)^{m+1} (-x)^{m}\sin(\rho m \pi)}{m\sin(\frac{m\pi}{\alpha})}
-\sum_{k=1}^\infty \frac{(-1)^{k+1} e^{i\alpha k\pi  }x^{\alpha k}\sin(\rho\alpha k\pi)}{k\sin(\alpha k\pi)}.
$$
The last limit is justified by  a standard estimation argument, that implies that in a converging power series one can enter the limit under the series. Moreover, the same argument shows that 
  when $0<x<1$, we have
  \begin{equation}\label{granica}
  l(x)= - \frac1\pi\lim_{w\to -x,\Im w>0} \Im G(w)= - \frac1\pi e^{\Re(w)} \sin({\Im(w)}). 
  \end{equation}
Now we evaluate
\begin{align*}
\Re(w)&=\sum_{m=1}^\infty \frac{x^{m}\sin(\rho m \pi)}{m\sin(\frac{m\pi}{\alpha})}
-\sum_{k=1}^\infty \frac{(-1)^{k+1} \cos(\alpha k\pi) x^{\alpha k}\sin(\rho\alpha k\pi)}{k\sin(\alpha k\pi)},\\
\Im(w)&=-
\sum_{k=1}^\infty \frac{(-1)^{k+1} \sin(\alpha k\pi) x^{\alpha k}\sin(\rho\alpha k\pi)}{k\sin(\alpha k\pi)}=
\sum_{k=1}^\infty \frac{(-1)^{k}  x^{\alpha k}\sin(\rho\alpha k\pi)}{k}.
\end{align*}
We will need the following formulas from \cite{2007-ISG-IMR-eap}
\begin{equation}\label{eq:sumsin}
\sum_{k=1}^\infty \frac{p^k \sin(k\varphi)}{k} = \arctan \frac{p \sin \varphi}{1 - p\cos \varphi}\,,\qquad \varphi \in (0,2\pi),\,p^2 \le 1.
\end{equation}
\begin{equation}\label{eq:sumcos}
\sum_{k=1}^\infty \frac{p^k \cos(k\varphi)}{k} = -\frac{1}{2}\log (1 - 2 p \cos \varphi + p^2)\,,\qquad \varphi \in (0,2\pi),\,p^2 \le 1.
\end{equation}
Therefore applying a formula $\sin(\arctan u)=\frac{u}{\sqrt{1+u^2}}$ we get
\begin{align*}
\sin(\Im (w)) & =\sin(  \arctan \frac{-x^\alpha \sin (\rho\alpha \pi)}{1 + x^\alpha\cos (\rho\alpha \pi)}) \\
& = \frac{-\frac{x^\alpha \sin (\rho\alpha \pi)}{1 + x^\alpha\cos (\rho\alpha \pi)}}{\sqrt{1 + \left(\frac{x^\alpha \sin (\rho\alpha \pi)}{1 + x^\alpha\cos (\rho\alpha \pi)}\right)^2}} = \frac{-x^\alpha \sin (\rho\alpha \pi)}{\sqrt{x^{2\alpha} + 2 x^\alpha\cos (\rho\alpha \pi) +1}}.
\end{align*}
Now we compute $\Re(w)$. By (\ref{eq:sumcos}) we get
\begin{align*}
\Re(w) &= \sum_{m=1}^\infty \frac{(-1)^{m+1}x^{m}\sin((1-\rho) m \pi)}{m\sin(\frac{m\pi}{\alpha})}\\
&+\sum_{k=1}^\infty \frac{(-1)^{k+1} x^{\alpha k}\sin((1-\rho)\alpha k\pi)}{k\sin(\alpha k\pi)} - \sum_{k=1}^\infty \frac{(-1)^{k+1} x^{\alpha k} \cos(\rho\alpha k\pi) }{k}\\ 
& = \log \hkappa(1,x) - \frac{1}{2}\log (1 + 2 x^\alpha \cos(\rho\alpha\pi) + x^{2\alpha}).
\end{align*}
Hence 
\begin{align*}
-\frac{1}{\pi}e^{\Re(w)}\sin(\Im(w)) &= \frac{\sin (\rho\alpha \pi)}{\pi} \frac{x^\alpha  \hkappa(1,x) }{x^{2\alpha} + 2 x^\alpha\cos (\rho\alpha \pi) +1}= l(x)>0.
\end{align*}
By \cite[Lemma 5]{PG-TJ-aihp} we have for $\theta >1$
$$
\kappa(1,\theta) = \theta^{\alpha\rho} \kappa(1,1/\theta)
$$
and we use the same method and  (\ref{granica}) to obtain
\begin{equation}\label{granica2}
 l(x)= - \frac1\pi\lim_{w\to -x,\Im w>0} \Im G(w)= \frac{\sin (\rho\alpha \pi)}{\pi} \frac{x^\alpha  \hkappa(1,x) }{x^{2\alpha} + 2 x^\alpha\cos (\rho\alpha \pi) +1} 
\end{equation}
for $x>0, x\not=1$. 

 As the function $l(x)$ is continuous at $x=1$ and by \cite[p.205]{2006-SF-ptrf} 
 the limit $\lim_{w\to -1,\Im w>0} \Im G(w)$ exists, 
 it follows
that the   convergence  in (\ref{granica2})   holds also for $x=1$.

Let us now  justify the fact that the function $G(\theta)=1/\kappa(1,\theta)$
is a Stieltjes transform of a positive measure $\mu$ on $\RR^+$. We will check the conditions (S1-3)
given in the beginning of this section.

  The function $G(\theta)$ is strictly positive for $\theta\in (0,\infty)$ and it extends to 
a  holomorphic function $h_1(z)$ on $\CC \setminus \RR_-$.
Thus the conditions (S1) and (S2) are verified. In order to  justify (S3),
we  use the following  property that we proved above:   the harmonic function $-(1/\pi)\Im G$ extends continuously
to the closed upper half-space $\{\Im z\ge 0\}$
and its boundary values on $\RR$ are $  l(-x)>0$ when  $x<0$ and $0$ for $x\ge 0$. 
Taking into account the fact that
 $\lim_{|z|\to\infty}G(z)=0$ (\cite[p.205]{2006-SF-ptrf}), the maximum principle(\cite[1.10]{ABR})
 implies that  $ \Im G(z)\le 0$ on $\{\Im z> 0\}$ and (S3) also holds.

 It follows  that for a certain $a\ge 0$ we have $G(\theta)=a+\pS(l)(\theta)$. Considering
$\theta\to\infty$ we determine $a=0$.

 Finally consider any $\alpha \in (0,2]$. Since the Lebesgue measure of the set $\mathcal{L} \cup \QQ$ is $0$ we can take a sequence $\alpha_n$ tending to $\alpha$. Passing to the limit we obtain (\ref{eq:inv_kappa}) for all $\alpha \in (0,2]$.
\end{proof}
{\bf Remark}. Other proofs of the fact that  $1/\kappa(1,\theta)$
is the Stieltjes transform of a positive measure $\mu$ seem  possible, using  properties of Bernstein functions (\cite{2010-RS-RS-ZV-gsm}). 

We deduce immediately from Proposition \ref{prop:St} the following corollary.
\begin{cor}
For $\rho \in [0, 1) \setminus \{1- 1/\alpha \}$ we have 
\begin{equation}\label{eq:inv_hkappa}
\frac{\sin((1-\rho) \alpha  \pi)}{\pi} \int_0^\infty \frac{1}{x+\theta} \frac{x^{\alpha} \kappa(1,x)}{x^{2\alpha} +2x^\alpha\cos((1-\rho)\alpha\pi) +1}\, dx = \frac{1}{\hkappa(1,\theta)}.
\end{equation}
\end{cor}
\section{Laplace transform of $\tau$ and applications}

The following theorem is the main result of the article.
\begin{thm}\label{thm:main}

Let $X_t$ be a non-spectrally positive strictly $\alpha$-stable process on $\RR$. For any $t>0$ we have
\begin{equation}\label{formula:main}
\EE^1 e^{-t \tau} = \frac{\sin((1-\rho)\alpha\pi)}{\pi} \int_0^\infty e^{-t^{1/\alpha}x} \frac{x^{\alpha-1} \kappa(1,x)}{x^{2\alpha} + 2x^{\alpha}\cos((1-\rho)\alpha\pi) +1}\,dx .
\end{equation}
\end{thm}
{\bf Remark.} Observe that the only case excluded from the Theorem \ref{thm:main} is well known:
when $X_t$ is a spectrally positive $\alpha$-stable process starting from $X_0=x$, $1<\alpha<2$,
then $(\tau^x_{(0,\infty)})_{x>0}$ is a $1/\alpha$-stable subordinator and $\EE^1 e^{-t \tau}=  e^{-t^{1/\alpha}}$
(\cite{Bing} p.281).  When  $X_t$ is spectrally negative, the formula   (\ref{formula:main}) was obtained
recently by T.Simon(\cite{TS-em}).
\begin{proof}
We note that if $g(0) = 1$ then
\begin{equation}\label{eq:impli}
\int_0^\infty \frac{1}{x+t} f(x) dx = g(t)\quad
\Longrightarrow \quad\int_0^\infty \frac{1}{x+t} \frac{f(x)}{x} dx = \frac{1}{t} - \frac{g(t)}{t}.
\end{equation}
Indeed 
$$
1 = \int_0^\infty \frac{x+t}{x+t}\frac{f(x)}{x}dx = g(t) + \int_0^\infty \frac{t}{x+t} \frac{f(x)}{x} dx
$$
From \cite{2006-EAK-sv} we know that 
\begin{equation}\label{eq:secondfactor}
\int_0^\infty  e^{-\theta y}\EE^y e^{-\eta \tau}\,dy = \frac{1}{\theta} - \frac{\hkappa(\eta,0)}{\theta\hkappa(\eta,\theta)},
\end{equation}
where $\tau = \tau_{(0,\infty)}$ and $\eta, \theta >0$. Putting $\eta = 1$ and applying (\ref{eq:impli}) to (\ref{eq:inv_hkappa}) we get
\begin{align*}
&\int_0^\infty  e^{-\theta y}\EE^y e^{-\tau}\,dy  \\ 
= & \frac{\sin((1-\rho) \alpha  \pi)}{\pi} \int_0^\infty \frac{1}{x+\theta} \frac{x^{\alpha-1} \kappa(1,x)}{x^{2\alpha} +2x^\alpha\cos((1-\rho)\alpha\pi) +1}\, dx \\
= & \frac{\sin((1-\rho) \alpha  \pi)}{\pi} \int_0^\infty\int_0^\infty e^{-y(x+\theta)} \frac{x^{\alpha-1} \kappa(1,x)}{x^{2\alpha} +2x^\alpha\cos((1-\rho)\alpha\pi) +1}\, dy\, dx \\
= & \int_0^\infty e^{- \theta y} \left(\frac{\sin((1-\rho) \alpha  \pi)}{\pi} \int_0^\infty e^{-yx} \frac{x^{\alpha-1} \kappa(1,x)}{x^{2\alpha} +2x^\alpha\cos((1-\rho)\alpha\pi) +1}\, dx \right) dy.
\end{align*}
Therefore
$$
\EE^y e^{-\tau}= \frac{\sin((1-\rho) \alpha  \pi)}{\pi} \int_0^\infty e^{-yx} \frac{x^{\alpha-1} \kappa(1,x)}{x^{2\alpha} +2x^\alpha\cos((1-\rho)\alpha\pi) +1}\, dx
$$
and assertion of the theorem follows from scaling property $\EE^y e^{-\tau} = \EE^1 e^{-y^\alpha \tau}$ of stable processes.
\end{proof}
For $\alpha \ge 1$ we immediately obtain from Theorem \ref{thm:main} a formula for the density $h= h_{\alpha,\rho}$ of $\tau$ under $\PP^1$.
\begin{cor}\label{cor:alpha>1}
Let $\alpha \in (1,2)$. The density of $\tau=\tau_{(0,\infty)}$ under $\PP^0$ is given by 
\begin{equation}\label{eq:alpha>1}
h(s) = \frac{\sin((1-\rho)\alpha\pi)}{\pi}\int_0^\infty \eta_{1/\alpha}(x,s) \frac{x^{\alpha-1} \kappa(1,x)}{x^{2\alpha} + 2x^{\alpha}\cos((1-\rho)\alpha\pi) +1}\, dx,
\end{equation}
where $\eta_{\gamma}(t,x)$ is the transition density of $\gamma$-stable subordinator.
\begin{proof}
Since $\int_0^\infty e^{-xs} \eta_{\gamma}(t,x) dx = e^{-ts^\gamma}$ we obtain by Theorem \ref{thm:main}
\begin{align*}
&\int_0^\infty e^{-st} h_\alpha(s)\,ds \\
&= \frac{\sin((1-\rho)\alpha\pi)}{\pi}\int_0^\infty \int_0^\infty e^{-st}\eta_{1/\alpha}(x,s) \frac{x^{\alpha-1} \kappa(1,x)}{x^{2\alpha} + 2x^{\alpha}\cos((1-\rho)\alpha\pi) +1}\,ds\,dx\\
&=\int_0^\infty e^{-st} \frac{\sin((1-\rho)\alpha\pi)}{\pi}\int_0^\infty  \eta_{1/\alpha}(x,s) \frac{x^{\alpha-1} \kappa(1,x)}{x^{2\alpha} + 2x^{\alpha}\cos((1-\rho)\alpha\pi) +1}\,dx\,ds
\end{align*}
and the assertion follows.
\end{proof}
\end{cor}
\begin{cor}\label{cor:alpha=1}
 If $X_t$ is a Cauchy process on $\RR$ ($\alpha=1, \rho=1/2$) then
 the density of the exit time $\tau=\tau_{(0,\infty)}$ under $\PP^1$ is given by
 $$
 h(x) =\frac{1}{\pi} \frac{\kappa(1,x)}{x^{2} +1}, \ x\ge 0.
 $$
\end{cor}

\noindent{\bf Remark}. The above formula  for $h_{1,1/2}$ was obtained previously by Darling in \cite{1956-DAD-tams} (see also \cite[(7.13)]{KKMS}).\\\\
For  $\alpha\not =1$, Theorem  \ref{thm:main} gives interesting multiplicative convolution
relations verified by $\tau$. We present them in the following subsection.

\subsection{Interpretation in terms of multiplicative convolutions}

For a given strictly $\alpha$-stable process $X_t$ with $\rho \in [0,1) \setminus \{1-1/\alpha\}$ we define the following function on $\RR^+$ 
$$
m_{\alpha,\rho}(x)=\frac{\sin((1-\rho)\alpha\pi)}{\pi\alpha} \frac{\kappa(1,x^{1/\alpha})}{x^{2} + 2x\cos((1-\rho)\alpha\pi) +1}, \ x\ge 0.
$$
Observe that $m_{\alpha,\rho}(x)$ is a probability density on $\RR^+$. This follows from the formula
$$
\frac{\sin((1-\rho)\alpha\pi)}{\pi} \int_0^\infty  \frac{x^{\alpha-1} \kappa(1,x)}{x^{2\alpha} + 2x^{\alpha}\cos((1-\rho)\alpha\pi) +1}\,dx=1
$$
obtained from (\ref{formula:main}) when $t\to 0$ and by a change of variables $x=y^{1/\alpha}$.

Denote by $M_{\alpha,\rho}$ a positive random variable with  density $m_{\alpha,\rho}$. The variable $M_{\alpha,1/\alpha}$, $1<\alpha<2$ appeared for the first time in \cite{TS-em}
for the special case of a completely asymmetric $\alpha$-stable process, $1<\alpha<2$ (in our context a spectrally negative process), when
$\kappa(1,x^{1/\alpha})=1+x^{1/\alpha}$.

Let $\gamma\in(0,1)$ and $\eta_{\gamma}(t,x)$ be the transition density of $\gamma$-stable subordinator. Denote by $N(\gamma)$ a random variable with the density $\eta_{\gamma}(1,x)$, i.e. $\EE\exp(-xN(\gamma))=e^{-x^\gamma}$,
$x\ge 0$. 

Recall that if $Y$ and $Z$ are independent random variables on $(0,\infty)$ with densities $f$ and $g$ respectively,
then the multiplicative convolution $Y \times Z^p$ is a random variable with the density 
\begin{equation}\label{eq:splot}
\PP[Y\times Z^p\in dt]= \int_0^\infty f(\frac{t}{u^p}) g(u) \frac{du}{u^p}.
\end{equation}

\begin{cor}\label{cor:sploty}
(i) Let $1<\alpha <2$. Suppose that  random variables $M_{\alpha,\rho}$ and $N(1/\alpha)$ are independent and that $X_0=1$. We have
$$\tau \stackrel{d}{=} M_{\alpha,\rho}\times N(1/\alpha), $$
(ii) Let $0<\alpha<1$. Suppose that random variables $\tau$ and $N(\alpha)$ are independent. We have
$$
\tau \times N(\alpha)^\alpha \stackrel{d}{=}  M_{\alpha,\rho}.
$$ 
\end{cor}
\begin{proof}

Part (i) follows immediately from Corollary \ref{cor:alpha>1}. In (\ref{eq:alpha>1}) we substitute $x^\alpha =u$ and use the scaling property $\eta_{1/\alpha}(u^{1/\alpha},s) = u^{-1}\eta_{1/\alpha}(1,su^{-1})$.

In order to prove (ii), we use $e^{-ty}=\EE[-(ty)^{1/\alpha} N(\alpha)]$ in the left-hand side of (\ref{formula:main}), we apply
Fubini
and  change variables $x=y^{1/\alpha } u$. By unicity of the Laplace transform we get
$$
\int_0^\infty h_\alpha(\frac{x^\alpha}{u^\alpha})\eta_{\alpha}(u)\frac{du}{u^\alpha}=
\frac{\sin((1-\rho)\alpha\pi)}{\pi\alpha}   \frac{\kappa(1,x)}{x^{2\alpha} + 2x^{\alpha}\cos((1-\rho)\alpha\pi) +1}
$$
Replacing $x^\alpha$ by $x$ in the last formula and using (\ref{eq:splot}) ends the proof of (ii).
\end{proof}
\subsection{Application for   the Doney's class $\hat C_{1,1}$ }

Let $\alpha\in[\frac12,1)$ and $1-\rho = 1/\alpha-1$, i.e. we consider a process dual to a process from the class $C_{1,1}$ from
R. Doney's article \cite{1987-RAD-ap}.
We denote the class of such processes by $\hat C_{1,1}$.
In this case we have by \cite{1987-RAD-ap} or \cite{PG-TJ-aihp}
$$\kappa(1,x) = \frac{x^{2\alpha}-2x^\alpha \cos \alpha\pi +1}{1+x}.$$

We recall the definition of the confluent hypergeometric functions $_1F_1$ and $U$ (see \cite{1964-AS-book})
\begin{align*}
&_1F_1(a,b,z) = \sum_{k=0}^\infty \frac{(a)_k z^k}{(b)_k k!}\,, \quad z \in \CC,\\
&U(a,b,z)\\
 &= \frac{\pi}{\sin \pi b}\left(\frac{_1F_1(a,b,z)}{\Gamma(1+a-b)\Gamma(b)} - z^{1-b} \frac{_1F_1(1+a-b,2-b,z)}{\Gamma(a)\Gamma(2-b)} \right), \quad z \in \CC \setminus \RR_-,
\end{align*}
where $(a)_k = a(a+1)\ldots(a+k-1)$, $(a)_0=1$ is Pochhammer symbol. Using Theorem \ref{thm:main} we get the following formulas for Laplace and Stieltjes transforms of $\tau$ 

\begin{prop}\label{cor:Lap}
 If  $\alpha\in[\frac12,1)$ and $X_t\in \hat C_{1,1}$ then
\begin{align}
(i) \quad &\EE^1 e^{-\tau t}= \frac{\sin(\alpha\pi)}{\pi} \Gamma(\alpha)\Gamma(1-\alpha, t^{1/\alpha})e^{t^{1/\alpha}},  \label{eq:Lap1}\\
(ii) \quad &\EE^1 \frac{1}{x+\tau} =  \int_0^\infty e^{-u} \left(\frac{e^{u^{1/\alpha}x^{-1/\alpha}}}{x} - \frac{u^{1/\alpha-1}}{x^{1/\alpha}} \frac{_1F_1(1,2-\alpha,u^{1/\alpha}x^{-1/\alpha})}{\Gamma(2-b)} \right)du, \label{eq:Lap2}
\end{align}
where $\Gamma(a,z) = \int_z^\infty t^{a-1}e^{-t}\,dt$ is the incomplete Gamma function.
 \end{prop}
\begin{proof} To prove $(i)$ we use simple transformations of integrals
 \begin{align}
\EE^1 e^{-\tau t} &= \frac{\sin(\alpha(1/\alpha-1)\pi)}{\pi} \int_0^\infty  \frac{e^{-t^{1/\alpha}x} x^{\alpha-1} (x^{2\alpha}-2x^\alpha \cos \alpha\pi +1)}{(1+x)(x^{2\alpha} + 2x^{\alpha}\cos(\alpha(1/\alpha-1)\pi) +1)}\,dx \nonumber\\
& = \frac{\sin(\alpha\pi)}{\pi} \int_0^\infty  \frac{e^{-t^{1/\alpha}x} x^{\alpha-1}}{1+x}\,dx \label{eq:C11}\\
& = \frac{\sin(\alpha\pi)}{\pi} \int_0^\infty\int_0^\infty e^{-t^{1/\alpha}x} e^{-s(x+1)}x^{\alpha-1}\,ds\,dx \nonumber\\
& = \frac{\sin(\alpha\pi)}{\pi} \int_0^\infty e^{-s} \frac{\Gamma(\alpha)}{(t^{1/\alpha}+s)^{\alpha}}\,ds = \frac{\sin(\alpha\pi)\Gamma(\alpha)}{\pi} \int_{t^{1/\alpha}}^\infty e^{t^{1/\alpha}-r} \frac{1}{r^{\alpha}}\,dr \nonumber\\
& = \frac{\sin(\alpha\pi)}{\pi} \Gamma(\alpha)\Gamma(1-\alpha, t^{1/\alpha})e^{t^{1/\alpha}}.\nonumber
\end{align}
To prove $(ii)$ we use the following integral representation of $U$ (see \cite{1964-AS-book})
$$
U(a,b,z) = \frac{1}{\Gamma(a)}\int_0^\infty e^{-zt} t^{a-1}(1+t)^{b-a-1}\,dt, \qquad \RE z >0.
$$
Applying this to (\ref{eq:C11}) we get
\begin{align*}
\EE^1 \frac{1}{x+\tau} &= \frac{\sin(\alpha\pi)}{\pi} \int_0^\infty\int_0^\infty e^{-xt} e^{-t^{1/\alpha}s} \frac{s^{\alpha-1}}{1+s}\,ds\,dt \\
&= \frac{\sin(\alpha\pi)}{\pi} \int_0^\infty\int_0^\infty e^{-u} e^{-u^{1/\alpha}x^{-1/\alpha}s} \frac{s^{\alpha-1}}{x(1+s)}\,ds\,du \\
& = \frac{\sin(\alpha\pi)}{\pi} \int_0^\infty e^{-u} \Gamma(\alpha) \frac{U(\alpha,\alpha, u^{1/\alpha}x^{-1/\alpha})}{x}\,du.
\end{align*}
Hence
\begin{align*}
\EE^1 \frac{1}{x+\tau} &= \int_0^\infty \frac{e^{-u}}{x} \left(_1F_1(\alpha,\alpha,u^{1/\alpha}x^{-1/\alpha}) - \frac{u^{1/\alpha-1}}{x^{1/\alpha-1}} \frac{_1F_1(1,2-\alpha,u^{1/\alpha}x^{-1/\alpha})}{\Gamma(2-b)} \right)\,du\\
&=  \int_0^\infty e^{-u} \left(\frac{e^{u^{1/\alpha}x^{-1/\alpha}}}{x} - \frac{u^{1/\alpha-1}}{x^{1/\alpha}} \frac{_1F_1(1,2-\alpha,u^{1/\alpha}x^{-1/\alpha})}{\Gamma(2-b)} \right)\,du.
\end{align*}
 \end{proof}
It is possible to invert Stieltjes transform in (\ref{eq:Lap2}) and the resulting density of $\tau$ is given by a series either in $x$ or in $1/x$ (cf. \cite{AK-ap}). For $\alpha=2/3$ the process $X_t$ is symmetric and the density of $\tau$ has a nice integral representation involving hypergeometric function
$$
_1F_2(a; b,c; z) = \sum_{k=0}^\infty \frac{(a)_k z^k}{(b)_k(c)_k k!}, \qquad z \in \CC.
$$
\begin{cor}\label{cor:density}
Let $X_t$ be a symmetric $\frac23$-stable process on $\RR$ with $X_0=1$. Then the density of $\tau=\tau_{(0,\infty)}$
is given by the formula
\begin{equation}\label{form23}
 \PP^1(\tau \in dx)= \frac{1}{\pi} \int_0^\infty e^{-tx} \left(\sin(t^{3/2}) +  \frac{t^{1/2}{_1F_2}(1;2/3,7/6;-t^3/4)}{\Gamma(4/3)} \right)\,dt 
\end{equation}
and its Laplace transform is
\begin{equation*}
\EE^1 e^{-t \tau} = \frac{\sqrt{3}}{2\pi}\Gamma(2/3)\Gamma(1/3,  t^{3/2})e^{ t^{3/2}} \,, \qquad t >0. 
\end{equation*}
\end{cor}
\begin{proof}
The second part is a direct consequence of (\ref{eq:Lap1}). By (\ref{eq:Lap2})
\begin{align*}
\EE^1 \frac{1}{x+\tau} &=
 \int_0^\infty e^{-u} \left(\frac{e^{u^{3/2}x^{-3/2}}}{x} - \frac{u^{1/2}}{x^{3/2}} \frac{_1F_1(1,4/3,u^{3/2}x^{-3/2})}{\Gamma(4/3)} \right)\,du.
\end{align*}
Inverting Stieltjes transform we get
\begin{align*}
&\PP^1(\tau \in dx) \\
&= -\frac{1}{\pi} \IM \int_0^\infty e^{-u} \left(\frac{e^{u^{3/2}x^{-3/2}e^{-3i\pi/2}}}{-x} - \frac{u^{1/2}}{x^{3/2}e^{3i\pi/2}} \frac{_1F_1(1,4/3,u^{3/2}x^{-3/2}e^{-3i\pi/2})}{\Gamma(4/3)} \right)\,du \\
&= -\frac{1}{\pi} \IM \int_0^\infty e^{-u} \left(\frac{e^{u^{3/2}x^{-3/2}i}}{-x} - i\frac{u^{1/2}}{x^{3/2}} \frac{_1F_1(1,4/3,u^{3/2}x^{-3/2}i)}{\Gamma(4/3)} \right)\,du. \\
\end{align*}
Since 
\begin{align*}
\IM i _1F_1(1,4/3,y i) &= \RE \sum_{k=0}^\infty \frac{(1)_k(iy)^k}{k!(4/3)_k} = \sum_{k=0}^\infty \frac{(-y^2)^k}{(4/3)_{2k}}\\
 &= \sum_{k=0}^\infty \frac{(1)_k(-y^2)^k}{k! (4/6)_k (7/6)_k 2^{2k}} = {_1F_2}(1; 2/3,7/6; -y^2/4)
\end{align*}
we finally get
\begin{align*}
&\PP^1(\tau \in dx) \\
&= \frac{1}{\pi} \int_0^\infty e^{-u} \left(\frac{\sin(u^{3/2}x^{-3/2})}{x} + \frac{u^{1/2}}{x^{3/2}} \frac{_1F_2(1;2/3,7/6;-u^3x^{-3}/4)}{\Gamma(4/3)} \right)\,du \\
&= \frac{1}{\pi} \int_0^\infty e^{-tx} \left(\sin(t^{3/2}) +  \frac{t^{1/2}{_1F_2}(1;2/3,7/6;-t^3/4)}{\Gamma(4/3)} \right)\,dt .
\end{align*}
\end{proof}

{\bf Remark}.  It is possible to obtain the formula
 (\ref{form23}) from the results of M. Kwa\'snicki \cite{MK}.
 However, in order to do this, one has to make several non-elementary transformations of integrals and
our approach seems simpler than using \cite{MK}.\\


{\bf Acknowledgement}. We thank Christian Berg for discussions on the Stieltjes transform
and Thomas Simon for helpful comments and bibliographical indications.

\bibliographystyle{abbrv}
%

\bibliography{ETofSP_gr}

\begin{thebibliography}{10}

\bibitem{1964-AS-book}
M.~Abramowitz and I.~A. Stegun.
\newblock {\em Handbook of mathematical functions with formulas, graphs, and
  mathematical tables}, volume~55 of {\em National Bureau of Standards Applied
  Mathematics Series}.
\newblock 1964.

\bibitem{ABR}
S.~Axler, P.~Bourdon, and W.~Ramey.
\newblock {\em Harmonic function theory, {S}econd edition}.
\newblock Graduate {T}exts in {M}athematics vol.137. Springer, New York, 2001.

\bibitem{CB-LNM}
C.~Berg.
\newblock Quelques remarques sur le c\^one de {S}tieltjes.
\newblock In {\em Seminar on {P}otential {T}heory, {P}aris, {N}o. 5
  ({F}rench)}, volume 814 of {\em Lecture Notes in Math.}, pages 70--79.
  Springer, Berlin, 1980.

\bibitem{Bing}
N.~H. Bingham.
\newblock Maxima of sums of random variables and suprema of stable processes.
\newblock {\em Z. Wahrscheinlichkeitstheorie und Verw. Gebiete}, 26:273--296,
  1973.

\bibitem{1956-DAD-tams}
D.~A. Darling.
\newblock The maximum of sums of stable random variables.
\newblock {\em Trans. Amer. Math. Soc.}, 83:164--169, 1956.

\bibitem{Doney2}
R.~A. Doney.
\newblock On the maxima of random walks and stable processes and the arc-sine
  law.
\newblock {\em Bull. London Math. Soc.}, 19(2):177--182, 1987.

\bibitem{1987-RAD-ap}
R.~A. Doney.
\newblock On {W}iener-{H}opf factorisation and the distribution of extrema for
  certain stable processes.
\newblock {\em Ann. Probab.}, 15(4):1352--1362, 1987.

\bibitem{2006-SF-ptrf}
S.~Fourati.
\newblock Inversion de l'espace et du temps des processus de {L}\'evy stables.
\newblock {\em Probab. Theory Related Fields}, 135(2):201--215, 2006.

\bibitem{PG-TJ-aihp}
P.~Graczyk and T.~Jakubowski.
\newblock On {W}iener {H}opf factors of stable processes.
\newblock to appear in Ann. Inst. Henri Poincar\'e (B).

\bibitem{2007-ISG-IMR-eap}
I.~S. Gradshteyn and I.~M. Ryzhik.
\newblock {\em Table of integrals, series, and products}.
\newblock Elsevier/Academic Press, Amsterdam, seventh edition, 2007.

\bibitem{KKMS}
T.~Kulczycki, M.~Kwa{\'s}nicki, J.~Ma{\l}ecki, and A.~St{\'o}s.
\newblock Spectral properties of the {C}auchy process on half-line and
  interval.
\newblock {\em Proc. Lond. Math. Soc. (3)}, 101(2):589--622, 2010.

\bibitem{AK-ap}
A.~Kuznetsov.
\newblock On extrema of stable processes.
\newblock to appear in Ann. Probab.

\bibitem{MK}
M.~Kwa\'snicki.
\newblock Spectral analysis of subordinated {B}rownian motions in half-line.
\newblock submitted, 2010, available at http://arxiv.org/abs/1006.0524.

\bibitem{2006-EAK-sv}
A.~E. Kyprianou.
\newblock {\em Introductory lectures on fluctuations of {L}\'evy processes with
  applications}.
\newblock Universitext. Springer-Verlag, Berlin, 2006.

\bibitem{TS-prep1}
P.~Patie and T.~Simon.
\newblock Intertwining certain fractional derivatives.
\newblock Submitted, 2010, available at http://arxiv.org/abs/1008.4713, 2010.

\bibitem{2010-RS-RS-ZV-gsm}
R.~L. Schilling, R.~Song, and Z.~Vondra{\v{c}}ek.
\newblock {\em Bernstein functions, Theory and {A}pplications}, volume~37 of
  {\em de Gruyter Studies in Mathematics}.
\newblock Walter de Gruyter \& Co., Berlin, 2010.

\bibitem{TS-em}
T.~Simon.
\newblock Fonctions de {M}ittag-{L}effler et processus de {L}\'evy stables sans
  sauts n\'egatifs.
\newblock {\em Expo. Math.}, 28(3):290--298, 2010.

\bibitem{TS-prep2}
T.~Simon.
\newblock Hitting densities for spectrally positive stable processes.
\newblock Submitted, 2010, available at http://arxiv.org/abs/1002.1540, 2010.

\bibitem{Widder}
D.~Widder.
\newblock {\em The Laplace transform}.
\newblock Princeton Univ.Press, 1966.

\end{thebibliography}

{\noindent Piotr Graczyk 
\newline
    Laboratoire de Math\'ematiques LAREMA,  \newline
    Universit\'e d'Angers,
    2, boulevard Lavoisier,
    49045 Angers Cedex 01, France \vspace{5pt}}

 {\noindent  Tomasz  Jakubowski  \newline
    Institute of Mathematics and Computer Science \newline
    Wroclaw University of Technology\newline
    50--370 Wroc{\l}aw, Wybrze\.ze Wyspia\'nskiego 27,  Poland
\vspace{10pt}}
\newline
{graczyk@univ-angers.fr, Tomasz.Jakubowski@pwr.wroc.pl} 
 
\end{document}